\documentclass{amsart}

\usepackage{amsmath,amssymb,amsthm,tikz}
\newtheorem{theorem}{Theorem}

\newtheorem*{proposition}{Proposition}

\theoremstyle{remark}

\DeclareMathOperator{\Si}{Si}

\begin{document}

\title[]{Three Convolution Inequalities on the Real Line\\ with connections to additive combinatorics}
\keywords{Convolution, Additive Number Theory, Fourier transform.}
\subjclass[2010]{11B13, 26D10, 28A75} 
\thanks{S.S. is supported by the NSF (DMS-1763179) and the Alfred P. Sloan Foundation.}

\author[]{Richard C. Barnard}
\address{Department of Mathematics, Western Washington University, Bellingham, WA 98225}
\email{barnarr3@wwu.edu}

\author[]{Stefan Steinerberger}
\address{Department of Mathematics, Yale University, New Haven, CT 06511}
\email{stefan.steinerberger@yale.edu}

\begin{abstract} We discuss three convolution inequalities that are connected to additive combinatorics.
Cloninger and the second author showed that for nonnegative $f \in L^1(-1/4, 1/4)$,
$$ \max_{-1/2 \leq t \leq 1/2} \int_{\mathbb{R}}{f(t-x) f(x) dx} \geq 1.28 \left( \int_{-1/4}^{1/4}{f(x) dx}\right)^2$$
which is related to $g-$Sidon sets (1.28 cannot be replaced by 1.52).
 We prove a dual statement, related to difference bases, and show that for $f \in L^1(\mathbb{R})$,
$$ \min_{0 \leq t \leq 1}\int_{\mathbb{R}}{f(x) f(x+t) dx} \leq 0.42 \|f\|_{L^1}^2,$$
where the constant 1/2 is trivial, 0.42 cannot be replaced by 0.37. This suggests a natural conjecture about the asymptotic structure of $g-$difference bases. Finally, we show for all functions $f \in L^1(\mathbb{R}) \cap L^2(\mathbb{R})$,
$$ \int_{-\frac{1}{2}}^{\frac{1}{2}}{ \int_{\mathbb{R}}{f(x) f(x+t) dx}dt} \leq 0.91 \|f\|_{L^1}\|f\|_{L^2}.$$
\end{abstract}

\maketitle

\vspace{-25pt}

\section{Introduction}

We discuss three convolution inequalities on the real line; one is well known, the other two seem to be new. The common theme is that all of them are fairly trivial if we do not care about the optimal constant. The optimal constant encapsulates something difficult in additive combinatorics that is not well understood.

\subsection{The first inequality.} Our first inequality is valid for $f \in L^1(\mathbb{R}) \cap L^2(\mathbb{R})$.
Recall that Fubini's theorem shows that
$$ \int_{\mathbb{R}}  \int_{\mathbb{R}}{f(x) f(x+t) dx} dt \leq \|f\|_{L^1}^2$$
while the Cauchy-Schwarz inequality shows that for any $t \in \mathbb{R}$
$$  \int_{\mathbb{R}}{f(x) f(x+t) dx}  \leq \|f\|_{L^2}^2$$
with equality attained for $t=0$. We prove a result between these two statements.
\begin{theorem} Let $f \in L^1(\mathbb{R}) \cap L^2(\mathbb{R})$. Then
$$ \int_{-\frac{1}{2}}^{\frac{1}{2}}{ \int_{\mathbb{R}}{f(x) f(x+t) dx}dt} \leq 0.91 \|f\|_{L^1}\|f\|_{L^2}$$
Moreover, the constant cannot be replaced by 0.8.
\end{theorem}

\subsection{Second inequality.}
The second inequality deals with a fundamental question in probability theory: if $f$ is a probability density in $(-1/4, 1/4)$, then the convolution $f*f$ is a probability density in $(-1/2, 1/2)$.
This means that the maximal value of $f*f$ has to exceed 1. How big does it necessarily have to be?  Cilleruelo, Ruzsa \& Vinuesa \cite{cill2} showed that finding the optimal constant $c$
in the inequality
\begin{equation} \label{inq} \max_{-1/2 \leq t \leq 1/2} \int_{\mathbb{R}}{f(t-x) f(x) dx} \geq c \left( \int_{-1/4}^{1/4}{f(x) dx}\right)^2, \end{equation}
is equivalent to answering an old question in additive combinatorics about the behavior of $g-$Sidon sets.
A subset $A \subset \left\{1, 2, \dots, n\right\}$ is called $g-$Sidon if 
$$ | \left\{(a,b) \in A\times A: a+b=m\right\}| \leq g$$
for every $m$. The main question is: how large can these $g-$Sidon sets be for a given $n$? Let us denote the answer by
$$ \beta_g(n) := \max_{A \subset \left\{1, 2, \dots, N\right\} \atop \mbox{\tiny A is g-Sidon}}{|A|}.$$
Cilleruelo, Ruzsa \& Vinuesa \cite{cill2} have shown that
$$\underline{\sigma(g)}\sqrt{gn}(1-o(1)) \leq \beta_g(n) \leq \overline{\sigma(g)}\sqrt{gn}(1+o(1)),$$
where the $o(1)$ is with respect to $n$ and 
$$ \lim_{g \rightarrow \infty}{\underline{\sigma(g)}} = c = \lim_{g \rightarrow \infty}{\overline{\sigma(g)}}$$
for some universal constant $c \in \mathbb{R}$ which is exactly the sharp constant in \eqref{inq}. Several
arguments have been suggested, in particular
\begin{align*}
c &\geq 1\qquad &&\mbox{trivial}\\
&\geq 1.151 \qquad &&\mbox{Cilleruelo, Ruzsa \& Trujillo \cite{cill1}, 2002} \\
&\geq  1.178 &&\mbox{Green \cite{green}, 2001}\\
&\geq  1.183 &&\mbox{Martin \& O'Bryant \cite{martin}, 2007}\\
&\geq  1.251 &&\mbox{Yu \cite{yu}, 2007}\\
&\geq  1.263 &&\mbox{Martin \& O'Bryant \cite{martin2}, 2009}\\
&\geq  1.274 &&\mbox{Matolcsi \& Vinuesa \cite{mat}, 2010}\\
&\geq  1.28 &&\mbox{Cloninger \& Steinerberger \cite{alex}, 2017}
\end{align*}
Matolcsi \& Vinuesa \cite{mat} also construct an example showing that $c \leq 1.52$, one is perhaps inclined to believe that this upper bound is close to the truth. It is this fascinating connection between additive combinatorics and real analysis that motivated us to look for a dual inequality.

\subsection{Third inequality.} The third inequality is in a similar spirit to \eqref{inq} and motivated by a problem in additive combinatorics of a similar spirit.
It is not difficult to see (`half of Fubini') that for any $f \in L^1(\mathbb{R})$
$$ \int_{0}^{1}{\int_{\mathbb{R}} f(x) f(x+t) dx dt} \leq 0.5 \|f\|_{L^1}^2$$
and the constant is sharp.
However, once we replace the average in $t$ with the minimum, the constant can be universally improved, which is our main result.

\begin{theorem}  Let $f \in L^1(\mathbb{R})$. Then 
$$ \min_{0 \leq t \leq 1}{\int_{\mathbb{R}} f(x) f(x+t) dx} \leq  0.411\|f\|_{L^1}^2$$
and the constant cannot be replaced by $0.37$. 
\end{theorem}
This can be understood as the continuous analogue of a nice problem in additive combinatorics.
 We say that a set $A \subset \mathbb{Z}$ is a difference basis with respect to $n$ if
$$ \left\{1, \dots, n\right\} \subseteq A - A,$$
where $A - A = \left\{a_1 - a_2: a_1 \in A, a_2 \in A\right\}$. A natural question, going back to Redei \& Renyi \cite{red}, is to understand the minimal size of $A$. A trivial estimate is
$$ n = \# \left\{1, \dots, n\right\}  \leq \#((A-A) \cap \mathbb{N} ) \leq \binom{|A|}{2} \leq \frac{|A|^2}{2}$$
which shows $|A| \geq \sqrt{2 n}$. The best known results are, for $n$ large,
$$ \sqrt{2.435 n} \leq |A| \leq \sqrt{2.645n},$$
where the lower bound was recently found by Bernshteyn \& Tait \cite{bern} (improving on a 1955 bound of Leech \cite{leech}). The upper bound is a 1972 result of Golay \cite{golay}.
Golay's writes that the optimal constant ``will, undoubtedly, never be expressed in closed form''. The book of Bollobas \cite{boll} has a nice description of the problem. We also refer to
papers of Erd\H{o}s \& Gal \cite{erdos}, Haselgrove \& Leech \cite{haselgrove}. The problem has some importance in engineering, cf. the book of Pott, Kumaran, Helleseth \& Jungnickel \cite{book}.
One wonders whether the continuous analogue might also have applications.\\

Theorem 2 and its similarity to the $g-$Sidon sets suggests another natural question. Let us define a set $A \subset \mathbb{Z}$ to be a $g-$difference basis with respect to $n$ if, for all $1 \leq k \leq n$ the equation
$$ a_i - a_j = k \qquad\mbox{has at least}~g~\mbox{solutions}.$$
The natural question is how small can such a set can be? Let us define
$$ \gamma_g(n) := \min_{A \subset \mathbb{Z} \atop \mbox{\tiny A is g-difference basis}}{|A|}.$$
Analogous to the result of Cilleruelo, Ruzsa \& Vinuesa \cite{cill2}, we could possibly hope that
$$\underline{\sigma(g)}\sqrt{gn}(1-o(1)) \leq \gamma_g(n) \leq \overline{\sigma(g)}\sqrt{gn}(1+o(1)),$$
where the $o(1)$ is with respect to $n$, 
$$ \lim_{g \rightarrow \infty}{\underline{\sigma(g)}} = c = \lim_{g \rightarrow \infty}{\overline{\sigma(g)}},$$
and $c$ is the optimal constant in the inequality
$$ \min_{0 \leq t \leq 1}{\int_{\mathbb{R}} f(x) f(x+t) dx} \leq  \frac{1}{c}\|f\|_{L^1}^2.$$
It is an open question whether the optimal constant in Theorem 2 can be given in closed form.

\subsection{Two related open problems.} We conclude by describing two fascinating open problems that seem to be very related in spirit. The first seems to have been raised by Martin \& O'Bryant \cite[Conjecture 5.2.]{martin3} and asks whether there is a universal constant $c>0$ such that for all nonnegative $f \in L^1(\mathbb{R}) \cap L^2(\mathbb{R})$ there is an improved H\"older inequality
$$ \|f*f\|_{L^2}^2 \leq (1-c) \|f*f\|_{L^1} \|f*f\|_{L^{\infty}}.$$
They proposed that it might even be true for a rather large value of $c$, they proposed $c \sim 0.1174$. This was disproven by Matolcsi \& Vinuesa \cite{mat} who showed that necessarily $c \leq 0.1107$. We believe it to be a rather fascinating question. 
 The second problem may, at first glance, seem unrelated: we noted the similarity to Theorem 2 because in both proofs the constant
$$ \inf_{x \in \mathbb{R}}{\frac{\sin{x}}{x}} \sim -0.217234$$
appears naturally. The question goes back to Bourgain, Clozel \& Kahane \cite{BCK}:
If $f:\mathbb{R} \rightarrow \mathbb{R}$ is an even function such that  $f(0) \leq 0$ and $\widehat{f}(0) \leq 0$, then it is not possible for both $f$ and
$\widehat{f}$ to be positive outside an arbitrarily small neighborhood of the origin. 
Having $f$ even and real-valued guarantees that $\widehat{f}$ is real-valued and even. {The second condition yields
$$ 0 \geq \widehat{f}(0)  = \int_{-\infty}^\infty{f(x)dx} \qquad \mbox{and} \qquad 0 \geq f(0) = \int_{-\infty}^\infty{\widehat{f}(y)dy},$$
which implies that the quantities
$$A(f):=\inf~\{r>0: f(x)\geq 0 \textrm{ if } |x|>r\}$$
$$A(\widehat{f}):=\inf~\{r>0: \widehat{f}(y)\geq 0 \textrm{ if } |y|>r\}$$
are strictly positive (possibly $\infty$) unless $f \equiv 0$. There is a dilation
symmetry $x \rightarrow \lambda x$ having the reciprocal effect $y \rightarrow y/\lambda$ on the Fourier side. As a consequence, the product $A(f)A(\widehat{f})$ is
invariant under this group action and becomes a natural quantity to consider. They prove that
$$A(f)A(\widehat{f}) \geq 0.1687,$$
and $0.1687$ cannot be replaced by $0.41$. The lower bound here is given by
$$ 0.1687 \sim \frac{1}{4(1+\lambda)^2} \qquad \mbox{where} \qquad \lambda = -\inf_{x \in \mathbb{R}}{\frac{\sin{x}}{x}}.$$
Goncalves, Oliveira e Silva and the second author  \cite{felipe1} improved this to
$$A(f)A(\widehat{f})\geq 0.2025, $$
and $0.2025$ cannot be replaced by $0.353$. It was also shown in \cite{felipe1} that the sharp constant is assumed by a function and that this function has infinitely many double roots. While this question is still open, a sharp form of the inequality in $d=12$ dimensions has been established by Cohn \& Goncalves \cite{cohn} using modular forms. There is also at least a philosophical similarity to problems related to the 'unavoidable geometry of probability distributions' \cite{alon, dong, feige, feldheim, schultze, steini}.

\section{Proofs}

\subsection{Preliminaries.}
All three proofs are based on the Fourier Analysis and variants of the Hardy-Littlewood rearrangement inequality. We recall that the Hardy-Littlewood inequality states that for bounded, positive and decaying functions $f, g:\mathbb{R} \rightarrow \mathbb{R}$
$$ \int_{\mathbb{R}}{f(x)g(x)dx} \leq \int_{\mathbb{R}}{ f^*(x) g^*(x) dx},$$
where $f^*(x)$ is the symmetric decreasing rearrangement of a function $f$. If one were to draw a picture, it would show that the integral is maximized if $f$ is rearranged in such a way that the `big' parts of $f$ interact with the `big' parts of $g$ and the `small' parts of $f$ interact with the `small' parts of $g$. Over regions where one of the functions is negative, the reverse statement is true and integrals are maximized by matching big contributions of one with small contributions of the other.
In terms of Fourier Analysis, we will make use of the Wiener-Khintchine theorem: using Plancherel's identity, we see that
\begin{align*}
 \int_{\mathbb{R}}{f(x)f(x+t) dx} &= \left\langle f, f(\cdot + t) \right\rangle_{L^2} \\
&=  \left\langle \widehat{f}(\xi), e^{2 \pi i \xi t} \widehat{f}(\xi)~ \right\rangle_{L^2} = \int_{\mathbb{R}}{e^{-2 \pi i \xi t} |\widehat{f}(\xi)|^2 d\xi}.
\end{align*}
In particular, the auto-correlation cannot look like the
characteristic function of a set (which are the types of functions for which H\"older's inequality is sharp)
and it is not hard to see how these types of identities would enter. As a toy example, we show that the
auto-correlation cannot be close to $\chi_{[-1,1]}$. 
\begin{proposition} Let $f \in L^1(\mathbb{R}) \cap L^2(\mathbb{R})$. Then
$$ \left\| \int_{\mathbb{R}}{f(x)f(x+t) dx} - \chi_{[-1,1]}(t) \right\|_{L^2} \geq \frac{3}{10}.$$
\end{proposition}
\begin{proof} The proof is simple: we use that the Fourier transform is unitary and obtain that
$$ \left\| \int_{\mathbb{R}}{f(x)f(x+t) dx} - \chi_{[-1,1]}(t) \right\|_{L^2} = \left\|  |\widehat{f}(\xi)|^2 - \frac{\sin{(2\pi \xi)}}{\pi \xi} \right\|_{L^2}.$$
However, one of these functions is positive while the other one becomes negative. So we clearly have at least
$$ \left\| |\widehat{f}(\xi)|^2 - \frac{\sin{(2\pi \xi)}}{\pi \xi} \right\|^2_{L^2} \geq \int_{\mathbb{R}}{ \left(\frac{\sin{(2\pi \xi)}}{\pi \xi}\right)^2 \chi_{\frac{\sin{(2\pi \xi)}}{\pi \xi} \leq 0} d\xi} \geq 0.1.$$
\end{proof}

\subsection{Proof of Theorem 1}

\begin{proof} We may assume, without loss of generality, that $f \geq 0$ and, using the invariance under scaling, that
$$ \int_{-\frac{1}{2}}^{\frac{1}{2}}{ \int_{\mathbb{R}}{f(x) f(x+t) dx}dt} =1.$$
We note that this normalisation dictates that
\begin{align*}
 1=\int_{-\frac12}^{\frac12}{ \int_{\mathbb{R}}{f(x) f(x+t) dx}dt}  \leq \int_{\mathbb{R}}^{} \int_{\mathbb{R}}{f(x) f(x+t) dx} dt = \|f\|_{L^1}^2.
\end{align*}
We distinguish two cases: the first case is
$$ \frac{1}{2}\frac{\|f\|_{L^2}^2}{\|f\|_{L^1}^{2}} \geq 0.88 \qquad \mbox{in which case} \qquad \|f\|_{L^2}^2 \geq 1.76 \|f\|_{L^1}^2 \geq 1.76$$
and we have shown the desired inequality since then
$$ \int_{-\frac{1}{2}}^{\frac{1}{2}}{ \int_{\mathbb{R}}{f(x) f(x+t) dx}dt} =1 \leq \|f\|_{L^1} \frac{\| f\|_{L^2}}{\sqrt{1.76}} \leq 0.8 \|f\|_{L^1} \|f\|_{L^2}.$$

 It remains to deal with the case where the fraction is smaller than 0.88 and we
will assume this throughout the subsequent argument.
Our main ingredient will be the Fourier transform which we use in the normalization
$$\widehat{f}(\xi) = \int_{\mathbb{R}}{ e^{-2 \pi i x \xi} f(x) dx},$$
which is the normalisation that turns it into a unitary transformation on $L^2(\mathbb{R})$. We estimate
\begin{align*}
1 = \int_{-\frac{1}{2}}^{\frac{1}{2}} \int_{\mathbb{R}}{f(x) f(x+t) dx} dt &= \int_{-\frac{1}{2}}^{\frac{1}{2}} \int_{\mathbb{R}}{ e^{-2 \pi i t \xi} |\widehat{f}(\xi)|^2 d\xi} dt\\
&=\int_{\mathbb{R}}^{}{ \frac{\sin{(\pi \xi)}}{\pi \xi}  |\widehat{f}(\xi)|^2 d\xi},
\end{align*}
We use the Hardy-Littlewood rearrangement inequality
$$ \int_{\mathbb{R}}^{}{ \frac{\sin{(\pi \xi)}}{\pi \xi}  |\widehat{f}(\xi)|^2 d\xi} \leq \int_{\mathbb{R}}^{}{ \max\left\{0,\frac{\sin{(\pi \xi)}}{\pi \xi}\right\}^* |\widehat{f}^*(\xi)|^2 d\xi}.$$
The symmetric decreasing rearrangement of the sinc function has a particularly simple form around the origin since
$$  \max\left\{0,\frac{\sin{(\pi \xi)}}{\pi \xi}\right\}^* = \frac{\sin{(\pi \xi)}}{\pi \xi} \qquad \mbox{for}~~|\xi| \leq 0.88.$$
We now estimate the rearranged Fourier transform and note that
$$ |\widehat{f}(\xi)| =  \left|  \int_{\mathbb{R}}{ e^{-2 \pi i x \xi} f(x) dx} \right| \leq  \int_{\mathbb{R}}{ |f(x)| dx} = \|f\|_{L^1}.$$
We now assume that $\widehat{f}(\xi)$ is equal to this maximal value over a large interval centered at the origin (by Hardy-Littlewood, this is a bound from above). However, we also have that $\|f\|_{L^2} = \|\widehat{f}\|_{L^2}$ which means that the Fourier transform $\widehat{f}$ can only be of size $\|f\|_{L^1}$ on an interval of total length $J$ centered at the origin where $\|f\|_{L^1}^2 J = \|f\|_{L^2}^2$. Therefore, we have
\begin{align*}
1 \leq \int_{\mathbb{R}}^{}{ \frac{\sin{(\pi \xi)}}{\pi \xi}  |\widehat{f}(\xi)|^2 d\xi} &\leq \int_{-\frac{1}{2}\|f\|_{L^2}^2\|f\|_{L^1}^{-2}}^{\frac{1}{2}\|f\|_{L^2}^2\|f\|_{L^1}^{-2}}{  \frac{\sin{(\pi \xi)}}{\pi \xi} \|f\|_{L^1}^2 d\xi}\\
&=\|f\|_{L^1}^2  \int_{-\frac{1}{2}\|f\|_{L^2}^2\|f\|_{L^1}^{-2}}^{\frac{1}{2}\|f\|_{L^2}^2\|f\|_{L^1}^{-2}}{  \frac{\sin{(\pi \xi)}}{\pi \xi}d\xi}.
\end{align*}
Introducing the special function
$$ \Si(x) = \int_0^x \frac{\sin{(y)}}{y} dy,$$
we can rewrite the result of our argument as
$$ 1 \leq  \frac{2 \|f\|_{L^1}^2}{\pi} \Si\left( \frac{\pi}{2} \frac{\|f\|_{L^2}^2}{\|f\|_{L^1}^2}\right),$$
or, making use of $\|f\|_{L^1} \geq 1$, and setting $x = 2 \|f\|_{L^1}^2 \pi^{-1}$ and $y = \|f\|_{L^2}^2$,
$$ x \Si\left(\frac{y}{x}\right) \geq 1 \qquad \mbox{where} \quad x \geq \frac{2}{\pi} \quad \mbox{and} \quad y > 0.$$
If we can show that this implies $x \cdot y  \geq 0.78$, then this this implies the main result.  We will now show that
this is indeed the case. We start by arguing that it suffices to restrict to $2/\pi \leq x \leq 0.78$ since $|\Si(x)| \leq |x|$ implies
$$ \frac{1}{x} \leq \Si\left(\frac{y}{x}\right)  \leq \frac{y}{x}$$
and thus $y \geq 1$. We now define, for $x > 2/\pi$, the function $y(x)$ as the smallest positive number such that
$$ x \Si\left(\frac{y(x)}{x}\right) = 1.$$
A quick computation shows that $y(2/\pi) = 1.226$ and thus $(2/\pi) y(2/\pi) \geq 0.78076\dots$. Differentiating the equation leads to
$$  \Si\left(\frac{y(x)}{x}\right) + x \frac{\sin{(y(x)/x)}}{y(x)/x}\left( \frac{y'(x)}{x} - \frac{y(x)}{x^2}\right) = 0.$$
This is an ordinary differential equation for $y(x)$ that gives all the necessary information. In particular, it shows that $(x y(x))' = 0.39\dots$ for $x= 2/\pi$
and thus $x y(x)$ is monotonically increasing for $x$ close to $2/\pi$. Standard error estimates then give the desired result for the entire range $2/\pi \leq x \leq 0.78$.
It remains to provide a lower bound for the constant.
If we set
$$ f(x) = e^{-a x^2},$$
then a computation shows that
$$ \frac{\int_{-\frac{1}{2}}^{\frac{1}{2}}{ \int_{\mathbb{R}}{f(x) f(x+t) dx}dt}}{\|f\|_{L^1}\|f\|_{L^2}} = \frac{(2 \pi)^{1/4} \mbox{erf}\left(\frac{\sqrt{a}}{\sqrt{8}}\right)}{a^{1/4}}$$
which yields $0.793$ for $a=7.839$. We used this as a start for a local search for functions of the type $f(x) = e^{-a x^2}(b + cx)$ and found that
$(a,b,c) = (15, 0.51, 8.55)$ results in the value 0.802.
\end{proof}

\subsection{Proof of Theorem 2}
\begin{proof} 
Let $f \in L^1(\mathbb{R})$ and suppose that for all $0 \leq t \leq 1$
$$ \int_{\mathbb{R}}{f(x)f(x+t)dx} \geq \frac{1}{2}.$$
By symmetry, this shows that the function
$$ g(t) =  \int_{\mathbb{R}}{f(x)f(x+t)dx}$$
satisfies $g(t) \geq 1/2$ for all $-1 \leq t \leq 1$.
Trivially, $\int_{\mathbb{R}}{g(t) dt} = \|f\|_{L^1}^2$. We write
$$ g(t) = \frac{1}{2} \chi_{[-1,1]} + h(t) \qquad \mbox{for some}~h(t) \geq 0.$$
Recalling the Wiener-Khintchine theorem (see \S 2.1.), we have that for any $\xi$
\begin{align*}
 0 \leq |\widehat{f}(\xi)|^2 &= \int_{\mathbb{R}}{e^{-\xi \pi i t} g(t)dt} \\
 &= \int_{\mathbb{R}}{e^{-\xi \pi i t} \left(  \frac{1}{2} \chi_{[-1,1]} + h(t) \right) dt} \\
 &\leq \frac{\sin{(2 \pi \xi)}}{2 \pi \xi}+ \int_{\mathbb{R}}{ h(t)dt}.
 \end{align*}
Optimizing over $\xi$ shows that
$$\int_{\mathbb{R}}{h(t) dt} \geq  - \inf_{x}\frac{\sin{x}}{x}$$
and thus
$$\int_{\mathbb{R}}{g(t) dt} \geq  1- \inf_{x}\frac{\sin{x}}{x}$$
and therefore
$$\| f\|_{L^1(\mathbb{R})} \geq \sqrt{1 - \inf_{x}\frac{\sin{x}}{x}} \sim 1.10328.$$
\end{proof}

\textbf{Construction of an Example.}
It remains to construct an example showing that the optimal constant in Theorem 2 cannot be less than 0.37. In light of the conjecture
made after Theorem 2 about the asymptotic structure of $g-$difference bases, we believe that extremal functions for Theorem 2 might
tell us something about the asymptotic structure of difference bases as well and are thus of interest beyond the scope of Theorem 2.\\

We now outline the reasoning behind our construction. The fact that our construction gives a very good lower bound (0.37)
almost matching our upper bound (0.41) is a good indicator that constructions of the type we consider are promising.

\begin{figure}[h!]
\begin{tikzpicture}
\draw [thick] (0.5,0) -- (4.5,0);
\draw [thick] (0.3+0.5,0) to[out=40, in=140] (2.6+0.5,0);
\draw [thick] (1.3+0.5,0) to[out=40, in=140] (3.6+0.5,0);
\draw[thick]  (5,0) -- (8,0);
\draw [thick] (5.3,0) -- (5.4,1) -- (6.3, 1) -- (6.4, 0);
\draw [thick] (6.3,0) -- (6.4,1) -- (7.3, 1) -- (7.4, 0);
\draw[thick]  (8.5,0) -- (11.5,0);
\draw[thick] (9, 1) to[out=270, in=180] (10-0.5, 0.2) to[out=0, in=270] (10, 1);
\draw[thick] (10.4-0.5, 1) to[out=270, in=180] (10+0.4, 0.2) to[out=0, in=270] (10.5+0.4, 1);
\end{tikzpicture}
\caption{Left: if the function $f$ is very flat, then the product becomes small. We thus want to have our functions concentrated. Middle: if the functions
are concentrated, then shifting it a bit leads to very little overlap leading again to small integrals. Right: if our functions has an integrable singularity at the boundary of its support, then
the integral has a chance of remaining large.}
\end{figure}
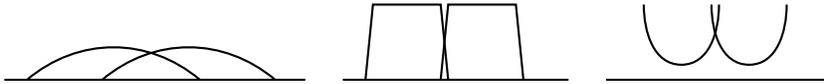
We are thus led to studying functions with compact support having a singularity on the boundary. A particularly natural such function is given by the
arcsine distribution $f:(-0.5, 0.5) \rightarrow \mathbb{R}$ given by
$$ f(x) = \frac{\chi_{[-0.5,0.5]}(x)}{\sqrt{1-4x^2}}.$$
which has many interesting properties (see e.g. \cite{coif} and references therein). In our normalization, we have
$$ \int_{\mathbb{R}}{f(x) dx} = \frac{\pi}{2}.$$
We now discuss one interesting property that we have not found in the literature and that inspired our example. For any $-1/2 < t < 1/2$, the integral
$$  \int_{\mathbb{R}}{f(x)f(x+t)dx} \qquad \mbox{is well defined}.$$
We note that for $t=0$, it is infinite (because the boundary singularities are of strength $\sim |1/2 - |x||^{-1/2}$), for all other values of $t$ the integral is finite (because
the boundary singularities by themselves are integrable). However, we would expect that the integral becomes small as $t \rightarrow  \pm 1$ (because that's when the
support of the two functions starts being very small). However, the scaling of the singularity is precisely cancels this effect!
We now carry out the associated computation.
\begin{align*}
 \int_{\mathbb{R}}{ f(x) f(x+t) dx} &= \int_{-\frac{1}{2} - t}^{\frac12}{ \frac{1}{\sqrt{1-4x^2}} \frac{1}{\sqrt{1-4(x+t)^2}}dx} \\
 &=\int_{-\frac{1}{2} - t}^{\frac12}{  \frac{1}{\sqrt{1+2x}} \frac{1}{\sqrt{1-2(x+t)}}  \frac{1}{\sqrt{1-2x}} \frac{1}{\sqrt{1+2(x+t)}} dx}.
 \end{align*}
We see that on the entire interval of integration
$$ \frac{1}{\sqrt{1+2x}} \geq \frac{1}{\sqrt{2}} \qquad \mbox{and} \qquad  \frac{1}{\sqrt{1-2(x+t)}} \geq \frac{1}{\sqrt{2}}$$
Therefore
$$  \int_{\mathbb{R}}{ f(x) f(x+t) dx}  \geq \frac{1}{2}  \int_{-\frac{1}{2} - t}^{\frac12}{  \frac{1}{\sqrt{1-2x}} \frac{1}{\sqrt{1+2(x+t)}} dx}.$$
Assuming $-1 < t < -1/2$, a linear substitution simplifies the integral to
$$ \int_{-\frac{1}{2} - t}^{\frac12}{  \frac{1}{\sqrt{1-2x}} \frac{1}{\sqrt{1+2(x+t)}} dx} = \frac{\pi}{2}.$$
and we have established 
$$  \int_{\mathbb{R}}{ f(x) f(x+t) dx} \geq \frac{\pi}{4} \qquad \mbox{for all}~0 < t < 1.$$

 It does not cover the case $t=1$ but this is not a big problem, our desired inequality is continuous under dilations
and we can thus consider, for any $\varepsilon > 0$, the function
$$ f_{\varepsilon}(x) = \frac{\chi_{[-0.5 -\varepsilon,0.5+\varepsilon]}}{\sqrt{1-\frac{x^2}{(0.5+\varepsilon)^2}}}$$
and choose $\varepsilon$ arbitrarily small. Actually computing
$$  \int_{\mathbb{R}}{ f(x) f(x+t) dx} \qquad \mbox{for various values of}~t$$
shows that the resulting function assumes its minimum on the boundary. This suggests that we can improve the function
a bit by removing some $L^1-$mass in the middle.
Numerically, for $\varepsilon$ sufficiently small
$$ f(x) = \frac{\chi_{[-0.5-\varepsilon,0.5+\varepsilon]}}{\sqrt{1-\frac{x^2}{(0.5 + \varepsilon)^2}}} - \frac{1}{4}\frac{\chi_{[-0.25,0.25]}}{\sqrt{1-4x^2}}  $$
we have
$$  \int_{\mathbb{R}}{f(x)f(x+t)dx} \geq \frac{\pi}{4} \qquad \mbox{while} \qquad \|f\|_{L^1}= 1.439$$
which shows that the constant cannot be less than 0.37. We note that, since we are removing $L^1-$mass in the middle, this does not affect any 
of the above computations for $t$ close to $\pm 1$. The remaining computations are all easy to do since everything is integrable.\\

We have found it difficult to attack the problem of constructing lower bounds numerically, this might be an interesting challenge. It would also be interesting to prove or disprove the existence of an extremizing function $f$ for this problem. We also note that having $f$ blow up like $\sim |1/2 - |x||^{-1/2}$ at the boundary seems optimal. If the blow-up is stronger, then the arising integral
is even bigger on the boundary: that does not help us but it does increase the $L^1-$mass which is not favorable. If the blow-up is weaker, then the integral tends to 0 which is also not favorable.
This motivates the following
\begin{quote}
\textbf{Open problem.} Is this reflected in the asymptotic structure of $g-$difference bases? Do they have an asymptotic density profile and does this density profile have $\sim |1/2 - |x||^{-1/2}$ blow-up at the boundary?
\end{quote}

\end{document}